\documentclass[a4paper,12pt]{amsart}

\usepackage[T1]{fontenc}
\usepackage{amsfonts}
\usepackage{amsmath}
\usepackage{amssymb}
\usepackage{mathrsfs}
\usepackage{microtype}
\usepackage{array}
\usepackage{booktabs}
\usepackage{float}
\usepackage{tabularx}
\usepackage[colorlinks=true,linkcolor=blue,citecolor=blue,urlcolor=blue]{hyperref}

\newcommand{\doi}[1]{\href{https://doi.org/#1}{doi:#1}}
\hypersetup{%
  pdftitle={Power-kernel fractional Sturm--Liouville operators: a graph realization and boundary-independent singular-value asymptotics},
  pdfauthor={Niyaz Tokmagambetov},
  pdfkeywords={fractional Sturm--Liouville operator, graph realization, boundary triplet, self-adjoint extension, Schatten class, Weyl asymptotics}
}

\newcolumntype{Y}{>{\raggedright\arraybackslash}X}

\numberwithin{equation}{section}
\theoremstyle{plain}
\newtheorem{thm}{Theorem}[section]
\newtheorem{prop}[thm]{Proposition}
\newtheorem{cor}[thm]{Corollary}
\newtheorem{lemma}[thm]{Lemma}

\newtheorem{property}[thm]{Property}

\theoremstyle{definition}
\newtheorem{defi}[thm]{Definition}
\newtheorem{rem}[thm]{Remark}
\newtheorem{ex}[thm]{Example}

\newcommand{\Dom}{\operatorname{Dom}}

\newcommand{\rank}{\operatorname{rank}}
\newcommand{\C}{\mathbb C}
\newcommand{\Dmax}{\mathfrak D_{\max}}

\newcommand{\Lmax}{\mathcal L_{M}}
\newcommand{\Lmin}{\mathcal L_{m}}
\newcommand{\Xiop}{\Xi}
\newcommand{\Gzero}{\Gamma_{0}}
\newcommand{\Gone}{\Gamma_{1}}
\newcommand{\Hspace}{L^2(0,1)}

\begin{document}

\title[Power-kernel fractional Sturm--Liouville operators]
{Power-kernel fractional Sturm--Liouville operators: a graph realization and boundary-independent singular-value asymptotics}

\author[N. Tokmagambetov]{Niyaz Tokmagambetov}
\address{Institute of Mathematics and Mathematical Modeling, 28 Shevchenko str., Almaty 050010, Kazakhstan}
\email{tokmagambetov@math.kz; niyaz.tokmagambetov@gmail.com}

\date{}


\begin{abstract}
We construct a closed graph realization in $L^2(0,1)$ for the power-kernel fractional Sturm--Liouville expression
\[
        \mathcal L u=\mathcal D_1^\alpha D_0^\alpha u,
        \qquad \frac12<\alpha<1.
\]
The graph domain contains the singular homogeneous mode $x^{\alpha-1}$, which belongs to $L^2(0,1)$ exactly in this range.  We first define the domain independently through the distributional Caputo composition and then prove the equivalent range characterization
\[
        I_0^{1-\alpha}u\in H^1(0,1),
        \qquad
        D_0^\alpha u-c\in\operatorname{ran}I_1^\alpha
\]
for some constant $c$.  We prove the Volterra representation of every vector in this domain, graph-norm completeness, and identify the graph operator as the adjoint of its closed densely defined zero-trace restriction.  Four bounded endpoint traces then form an ordinary boundary triplet.  Hence all self-adjoint extensions of the zero-trace operator, including coupled endpoint conditions, are parametrized by Lagrangian two-dimensional boundary subspaces.

For each self-adjoint realization we give an exact zero-eigenvalue criterion and an explicit inverse formula whenever zero is regular.  The inverse is $I_0^\alpha I_1^\alpha$ plus a finite-dimensional correction whose exact rank is read off from the boundary matrix.  For every admissible self-adjoint realization satisfying
$0\in\rho(\mathcal L_\omega)$, sharp singular-value asymptotics for fractional integration yield
\[
        s_n(\mathcal L_\omega^{-1})
        =(\pi n)^{-2\alpha}\bigl(1+O(n^{-1})\bigr),
        \qquad n\to\infty.
\]
Consequently,
\[
        \mathcal L_\omega^{-1}\in S_p
        \quad\Longleftrightarrow\quad
        p>\frac1{2\alpha},
\]
with the corresponding weak-Schatten endpoint and Weyl counting asymptotics for $|\mathcal L_\omega|$.  The contribution of this spectral statement is its validity for every invertible self-adjoint boundary plane; sharper phase information was previously known for the positive reference boundary problem.  The paper also records improvements to earlier power-kernel formulas concerning endpoint signs, moment functionals, homogeneous corrections, zero modes and the Schatten threshold.  As an application, the positive reference realization is used to state a Caputo-time diffusion problem as a mild-solution theorem in $C([0,T];L^2(0,1))$.
\end{abstract}

\maketitle

\section{Introduction}

Fractional diffusion models use nonlocal kernels to encode memory and anomalous transport.  The general fractional-dynamics background is surveyed in \cite{MK00}; models in heterogeneous media separated by membranes are discussed in \cite{LRTZE16}; and the broader membrane-science context can be found in \cite{AT13,H96}.  The present paper does not develop a new physical membrane model.  Its purpose is to give a precise Hilbert-space realization of the one-dimensional spatial operator that appears in such fractional diffusion equations, with particular attention to domains, endpoint traces and self-adjoint boundary conditions.

We use the standard Riemann--Liouville and Caputo fractional operators in the notation of \cite{SKM93,KST06}.  The formal expression studied in the paper is
\begin{equation}\label{EQ:L-intro}
        \mathcal L u(x):=\mathcal D_1^\alpha\left[D_0^\alpha u\right](x),
        \qquad 0<x<1, \qquad \frac12<\alpha<1,
\end{equation}
where $D_0^\alpha$ is the left Riemann--Liouville derivative and $\mathcal D_1^\alpha$ is the right Caputo derivative.  The range $\alpha>1/2$ is structural: it is exactly the condition under which
\[
        \frac{x^{\alpha-1}}{\Gamma(\alpha)}\in L^2(0,1),
        \qquad
        f\mapsto \int_0^1s^{\alpha-1}f(s)\,ds
        \quad\text{is bounded on }L^2(0,1).
\]
Thus the singular homogeneous mode is a genuine $L^2$ boundary mode in precisely this range.

The power-kernel expression \eqref{EQ:L-intro}, its Green formula and classes of
self-adjoint boundary conditions were studied in
\cite{TT16,TT18G,TT18M,TT19}.  In particular, the endpoint quantities and
matrix description of extensions have antecedents in those papers; part of the
purpose here is to place the improved formulas on a closed $L^2$ graph domain
and to prove the adjoint and boundary-triplet statements without relying on a
formal endpoint calculation.  Other fractional Sturm--Liouville formulations,
including Robin, integral and numerical approaches, appear in
\cite{KA13,Kli19,KCB18,Kli21,ZK13}.  The positive reference problem associated
with the compact operator $I_0^\alpha I_1^\alpha$ was also analyzed in
\cite{CK21}, where a sharper phase asymptotic was obtained.  Recent related
developments include Al-Jararha, Al-Refai and Luchko \cite{AAL24}, who construct
a self-adjoint fractional Sturm--Liouville problem for general fractional
derivatives generated by Sonin kernels, and Frymark and Liaw \cite{FL23}, who
use boundary triples and perturbation theory for singular classical
Sturm--Liouville operators with two limit-circle endpoints.  Against this
background, the contribution of the present paper is a distributionally
defined closed graph domain for the specific left Riemann--Liouville/right
Caputo expression, a rigorous ordinary boundary triplet on that domain, exact
finite-dimensional zero-mode and inverse matrices, and a common leading
singular-value and Weyl law for all invertible self-adjoint boundary planes.

The graph-domain point of view adopted here differs from imposing classical
endpoint values: ordinary endpoint traces do not see the singular mode
$x^{\alpha-1}$ correctly.  The representation proved below is
\begin{equation}\label{intro-graph-representation}
        u=I_0^\alpha I_1^\alpha f+C_0\frac{x^{\alpha-1}}{\Gamma(\alpha)}
          +C_1\frac{x^\alpha}{\Gamma(\alpha+1)},
        \qquad f\in L^2(0,1),\quad C_0,C_1\in\C.
\end{equation}
This representation separates the particular Volterra part from the two homogeneous modes before boundary conditions are imposed.  It also produces four bounded graph traces, namely the canonical endpoint values of $I_0^{1-\alpha}u$ and $D_0^\alpha u$.

The first main result is a distributional domain theorem.  Starting without the
ansatz \eqref{intro-graph-representation}, set
$F=I_0^{1-\alpha}u$.  We require $F\in H^1(0,1)$ and, for some $c\in\C$,
\[
 I_1^{1-\alpha}(F'-c)\in H^1(0,1),
 \qquad \bigl(I_1^{1-\alpha}(F'-c)\bigr)(1)=0.
\]
This is the endpoint-normalized weak graph of the right Caputo derivative.
We prove that $c$ is unique, that $F'$ then has a continuous representative
with $F'(1)=c$, and that this domain is
exactly \eqref{intro-graph-representation}, or equivalently the range domain
\[
 I_0^{1-\alpha}u\in H^1(0,1),\qquad
 D_0^\alpha u-c\in\operatorname{ran}I_1^\alpha
 \quad\text{for some }c\in\C.
\]
Thus the Volterra representation is a theorem rather than the definition of
the domain.  The graph operator is subsequently identified as the adjoint of
its closed densely defined zero-trace restriction.  The Green formula is then
derived on this graph domain and written in boundary coordinates.

The second main result is the extension theory.  Using the ordinary boundary-triplet framework of \cite{GG91,BHdS20}, the four graph traces yield a boundary triplet for the maximal operator.  Hence self-adjoint realizations are exactly the restrictions
\[
        \omega\Xi u=0,
        \qquad \omega\in\C^{2\times4},\quad
        \rank\omega=2,
        \quad \omega J\omega^*=0,
\]
that is, the Lagrangian two-dimensional boundary subspaces of the finite-dimensional boundary space.  This includes separated and coupled endpoint conditions.

The third main result is an explicit inverse and zero-mode criterion.  For every admissible boundary matrix $\omega$, zero is regular precisely when a $2\times2$ matrix $M_\omega$ is invertible.  In that case
\[
        \mathcal L_\omega^{-1}=I_0^\alpha I_1^\alpha+F_\omega,
        \qquad \rank F_\omega\le2,
\]
and the exact rank of $F_\omega$ is determined by two columns of $\omega$.  This makes the zero-eigenvalue obstruction and the boundary-dependent homogeneous correction explicit.

The fourth main result concerns the boundary-independent leading spectral law.
The singular-value theory of Riemann--Liouville fractional integration was
developed in \cite{FW86,Dos93,TG94}; for the range used here the estimate needed
below is Burman's sharper form \cite{Bur07},
\begin{equation}\label{intro-burman-asymptotics}
        s_n(K_\alpha)=(\pi n)^{-\alpha}\bigl(1+O(n^{-1})\bigr).
\end{equation}
Combining \eqref{intro-burman-asymptotics} with the identity
$ I_0^\alpha I_1^\alpha=(I_1^\alpha)^*I_1^\alpha $ and finite-rank singular-value inequalities from trace-ideal theory \cite{GK69,Sim05}, we obtain
\[
        s_n(\mathcal L_\omega^{-1})
        =(\pi n)^{-2\alpha}\bigl(1+O(n^{-1})\bigr),
\]
for every invertible self-adjoint realization.  In particular,
\[
        \mathcal L_\omega^{-1}\in S_p
        \quad\Longleftrightarrow\quad
        p>\frac1{2\alpha},
\]
with the weak-Schatten endpoint and the associated Weyl counting constant.
For the reference boundary conditions, the inverse is
$I_0^\alpha I_1^\alpha$.  The corresponding reference eigenproblem is the
problem denoted by $P'$ in \cite{CK21}, where the sharper phase formula
\[
 \lambda_n^{1/(2\alpha)}=\pi n-\frac{\pi}{2}+O(n^{-1})
\]
is recorded.  The new content of the present spectral argument is therefore not a sharper reference phase, but the transfer of the leading constant, the optimal Schatten threshold and the weak endpoint to every invertible self-adjoint boundary plane through the rank-two inverse correction.

Finally, the positive reference extension is used as the spatial operator in a Caputo-time problem
\begin{equation}\label{EQ:Anomalous-Diffusion-00}
{}^C D_{0+,t}^{\beta}u(t)+Au(t)=f(t),
        \qquad 0<\beta<1.
\end{equation}
The mild solution is stated by spectral expansion, following the standard fractional-evolution and eigenfunction-expansion approach in \cite{Baz01,SY11}.

\section{Preliminaries and a graph-domain lemma}\label{Properties}

Throughout the paper the Hilbert space is $\Hspace$, with inner product
\[
        (u,v)=\int_0^1u(x)\overline{v(x)}\,dx,
\]
linear in the first argument.  For $p>0$, $S_p$ denotes the Schatten--von Neumann ideal; for $0<p<1$ it is understood in the standard quasi-normed sense.  We also write $S_{p,\infty}$ for the weak Schatten ideal, defined by $\sup_{n\ge1} n^{1/p}s_n(A)<\infty$.  We use the standard trace-ideal notation and singular-value inequalities from \cite{GK69,Sim05}.

We use standard facts on Riemann--Liouville integrals and derivatives, including the semigroup property and the left--right integration-by-parts formula, from \cite{SKM93,N03,KST06}.

\begin{defi}
Let $f$ be a function on $[0,1]$ and let $\alpha>0$.  The left and right Riemann--Liouville fractional integrals are
\[
I_0^\alpha f(t)=\frac1{\Gamma(\alpha)}\int_0^t(t-s)^{\alpha-1}f(s)\,ds,
\qquad
I_1^\alpha f(t)=\frac1{\Gamma(\alpha)}\int_t^1(s-t)^{\alpha-1}f(s)\,ds.
\]
\end{defi}

\begin{defi}
For $0<\alpha<1$, the left and right Riemann--Liouville derivatives are
\[
D_0^\alpha f=\frac{d}{dt}I_0^{1-\alpha}f,
        \qquad
D_1^\alpha f=-\frac{d}{dt}I_1^{1-\alpha}f.
\]
The corresponding Caputo derivatives are
\[
\mathcal D_0^\alpha f=D_0^\alpha[f-f(0)],
        \qquad
\mathcal D_1^\alpha f=D_1^\alpha[f-f(1)],
\]
whenever the endpoint values and the displayed expressions are defined.
\end{defi}

Unless a classical representative is explicitly specified,
Riemann--Liouville derivatives are understood distributionally on $(0,1)$.
Endpoint values are used only for the continuous or $H^1$ representatives
constructed below; no endpoint value is assigned to a generic $L^2$ function.

\begin{property}[Adjointness of the left and right fractional integrals]\label{l6}
For $u,v\in L^2(0,1)$ and $\beta>0$,
\begin{equation}\label{IBP-integrals}
        (I_1^\beta u,v)=(u,I_0^\beta v).
\end{equation}
\end{property}
This is the standard left--right duality relation for Riemann--Liouville fractional integrals; it follows by Fubini's theorem and is recorded, for example, in \cite{SKM93,KST06}.

The following Riemann--Liouville inversion identities are used repeatedly
below; see \cite{SKM93,KST06}:
\begin{equation}\label{fractional-basic-identities}
        D_0^\alpha I_0^\alpha f=f,
        \qquad
        D_1^\alpha I_1^\alpha f=f
\end{equation}
in the distributional sense whenever $f\in L^2(0,1)$.  In particular,
\begin{equation}\label{T-identity}
        D_0^\alpha I_0^\alpha I_1^\alpha f=I_1^\alpha f.
\end{equation}
The corresponding Caputo identity requires an endpoint value and will be used
only after the restriction $\alpha>1/2$ and the endpoint estimate
\eqref{right-integral-endpoint-estimate} have been established.

The next elementary lemma will be used three times.

\begin{lemma}[Canonical graph lemma]\label{abstract-graph-lemma}
Let $H$ be a Hilbert space, let $T\in\mathcal B(H)$, and let $\phi_0,\phi_1\in H$ be linearly independent.  Assume that every element of
\[
        \mathfrak D=\{Tf+C_0\phi_0+C_1\phi_1:
        f\in H,\ C_0,C_1\in\C\}
\]
has a unique such representation.  Define
\[
        A(Tf+C_0\phi_0+C_1\phi_1)=f.
\]
Then the graph norm of $A$ is equivalent to the product norm of $H\oplus\C^2$ under the representation map.  In particular, $A$ is closed and $(\mathfrak D,\|\cdot\|_A)$ is a Hilbert space.  Every linear functional depending continuously on $f,C_0,C_1$ is bounded in the graph norm.
\end{lemma}

\begin{proof}
The upper estimate
\[
 \|Tf+C_0\phi_0+C_1\phi_1\|^2+\|f\|^2
 \le C\bigl(\|f\|^2+|C_0|^2+|C_1|^2\bigr)
\]
is immediate.  Since $\phi_0$ and $\phi_1$ are linearly independent, all norms on their span are equivalent; hence
\[
 |C_0|+|C_1|
 \le C\|C_0\phi_0+C_1\phi_1\|
 \le C\bigl(\|Tf+C_0\phi_0+C_1\phi_1\|+\|T\|\|f\|\bigr).
\]
Together with the graph-norm term $\|f\|$, this gives the reverse estimate.  Completeness follows from the completeness of $H\oplus\C^2$, and closedness of $A$ is equivalent to completeness of its graph domain.
\end{proof}

\section{The power-kernel graph realization and Green's formula}\label{sec:power-graph}

For the remainder of Sections \ref{sec:power-graph}--\ref{sec:comparison-previous} we assume
\[
        \frac12<\alpha<1
\]
and consider the formal expression
\begin{equation}\label{FDO-in-L2}
        \mathcal L u=\mathcal D_1^\alpha D_0^\alpha u.
\end{equation}
Set
\begin{equation}\label{T-and-modes-power}
        Tf:=I_0^\alpha I_1^\alpha f,
        \qquad
        \phi_0(x):=\frac{x^{\alpha-1}}{\Gamma(\alpha)},
        \qquad
        \phi_1(x):=\frac{x^\alpha}{\Gamma(\alpha+1)}.
\end{equation}
Then $\phi_0,\phi_1\in L^2(0,1)$.  The standard power-function formulas for Riemann--Liouville derivatives \cite{SKM93,KST06} give
\begin{equation}\label{power-homogeneous-identities}
        D_0^\alpha\phi_0=0,
        \qquad
        D_0^\alpha\phi_1=1,
        \qquad
        \mathcal L\phi_0=\mathcal L\phi_1=0.
\end{equation}
For $f\in L^2(0,1)$ define
\begin{equation}\label{power-moments}
        A_f:=\frac1{\Gamma(\alpha)}\int_0^1s^{\alpha-1}f(s)\,ds,
        \qquad
        B_f:=\frac1{\Gamma(\alpha+1)}\int_0^1s^\alpha f(s)\,ds.
\end{equation}
The assumption $\alpha>1/2$ makes both functionals continuous on $L^2(0,1)$ by Cauchy--Schwarz.
For later use we record that, for every $f\in L^2(0,1)$, $I_1^\alpha f$ has a continuous representative on $[0,1]$ and its canonical right endpoint value is zero.  Indeed, after zero extension and reflection it is a convolution of two $L^2(\mathbb R)$ functions, hence continuous by the elementary $L^2*L^2\subset C_0$ convolution theorem; moreover, for $0\le x<1$,
\begin{equation}\label{right-integral-endpoint-estimate}
 |I_1^\alpha f(x)|
 \le \frac{\|f\|_2}{\Gamma(\alpha)}
      \left(\int_x^1(s-x)^{2\alpha-2}\,ds\right)^{1/2}
 =\frac{(1-x)^{\alpha-1/2}}{\Gamma(\alpha)\sqrt{2\alpha-1}}\,\|f\|_2
 \to0
\end{equation}
as $x\uparrow1$.
Consequently, in the range $\alpha>1/2$ used below,
\begin{equation}\label{right-Caputo-inversion}
 \mathcal D_1^\alpha I_1^\alpha f
 =D_1^\alpha\bigl[I_1^\alpha f-(I_1^\alpha f)(1)\bigr]
 =D_1^\alpha I_1^\alpha f=f,
 \qquad f\in L^2(0,1).
\end{equation}

\begin{rem}[Sharpness of the range $\alpha>1/2$]\label{sharp-alpha-range}
The assumption $\alpha>1/2$ is not a technical convenience.  It is exactly the condition under which
\[
        x^{\alpha-1}\in L^2(0,1)
\]
and also exactly the condition under which the moment functional
\[
        f\mapsto\int_0^1s^{\alpha-1}f(s)\,ds
\]
is bounded on $L^2(0,1)$.  Hence the four-coordinate graph trace used in this paper is specific to the range $\alpha>1/2$.  At and below $\alpha=1/2$, the singular boundary mode leaves $L^2(0,1)$ and the extension theory must be reformulated with a different maximal domain.
\end{rem}

Put
\[
 H^1_{\{1\}}(0,1):=\{w\in H^1(0,1):w(1)=0\}.
\]

\begin{defi}[Endpoint-normalized maximal right-Caputo operator]
\label{maximal-right-caputo-definition}
A function $g\in L^2(0,1)$ belongs to
$\Dom(\mathcal C_{1,\max}^\alpha)$ if there is a constant $c\in\C$ such that
\begin{equation}\label{weak-right-caputo-domain}
 I_1^{1-\alpha}(g-c)\in H^1_{\{1\}}(0,1).
\end{equation}
For such a function set
\begin{equation}\label{weak-right-caputo-action}
 \mathcal C_{1,\max}^\alpha g
 :=-\frac{d}{dx}I_1^{1-\alpha}(g-c).
\end{equation}
The next lemma proves that $c$, and hence the action, is unique.
\end{defi}

\begin{lemma}[Weak right-Caputo inversion]\label{weak-right-caputo-inversion}
Let $1/2<\alpha<1$.  Then
$g\in\Dom(\mathcal C_{1,\max}^\alpha)$ if and only if there are unique
$c\in\C$ and $f\in L^2(0,1)$ such that
\begin{equation}\label{right-caputo-volterra-representation}
 g=c+I_1^\alpha f.
\end{equation}
In that case $g$ has a continuous representative on $[0,1]$ and
\[
 g(1)=c,
 \qquad \mathcal C_{1,\max}^\alpha g=f
 =D_1^\alpha[g-g(1)].
\]
\end{lemma}

\begin{proof}
First, the constant in \eqref{weak-right-caputo-domain} is unique.  If
$c_1$ and $c_2$ are both admissible, then
\[
 I_1^{1-\alpha}(c_2-c_1)
 =\frac{c_2-c_1}{\Gamma(2-\alpha)}(1-x)^{1-\alpha}
 \in H^1(0,1).
\]
For $c_1\ne c_2$, its weak derivative is a nonzero constant multiple of
$(1-x)^{-\alpha}$, which is not in $L^2(0,1)$ because $\alpha>1/2$.
Thus $c_1=c_2$.

Let $w=I_1^{1-\alpha}(g-c)\in H^1_{\{1\}}(0,1)$ and put
$f=-w'\in L^2(0,1)$.  Since $w(1)=0$, the fundamental theorem for $H^1$
and the semigroup property give
\[
 w(x)=\int_x^1f(s)\,ds=I_1^1f(x)
     =I_1^{1-\alpha}I_1^\alpha f(x).
\]
The operator $I_1^{1-\alpha}$ is injective on $L^2(0,1)$: applying
$I_1^\alpha$ to $I_1^{1-\alpha}h=0$ gives $I_1^1h=0$, and
distributional differentiation gives $h=0$.  Hence
$g-c=I_1^\alpha f$.

Conversely, if $g=c+I_1^\alpha f$, then
\[
 I_1^{1-\alpha}(g-c)=I_1^1f\in H^1_{\{1\}}(0,1),
 \qquad -\frac{d}{dx}I_1^1f=f.
\]
The continuity and endpoint assertion follow from
\eqref{right-integral-endpoint-estimate}; uniqueness of $f$ follows from
the injectivity of $I_1^\alpha$.
\end{proof}

\begin{defi}[Distributional graph domain]\label{power-maximal-domain-definition}
For $u\in L^2(0,1)$ put $F=I_0^{1-\alpha}u$.  Define
\begin{equation}\label{distributional-graph-domain}
 \Dmax:=\left\{u\in L^2(0,1):
 F\in H^1(0,1),\quad
 F'\in\Dom(\mathcal C_{1,\max}^\alpha)\right\}.
\end{equation}
On this domain set
\begin{equation}\label{power-graph-action}
 \Lmax u:=\mathcal C_{1,\max}^\alpha F'.
\end{equation}
Thus \eqref{distributional-graph-domain} is defined directly from weak
Riemann--Liouville differentiation and the endpoint-normalized right-Caputo
graph, not from a Volterra inverse formula.
\end{defi}

\begin{thm}[Graph representation and range characterization]
\label{power-natural-domain-characterization}
The distributional graph domain has the representation
\begin{equation}\label{Dmax-def}
 \Dmax=\left\{u=Tf+C_0\phi_0+C_1\phi_1:
 f\in L^2(0,1),\ C_0,C_1\in\C\right\}.
\end{equation}
The representation is unique and
\begin{equation}\label{power-natural-action}
 \Lmax(Tf+C_0\phi_0+C_1\phi_1)=f.
\end{equation}
Equivalently,
\begin{equation}\label{power-natural-domain}
\begin{aligned}
\Dmax=\mathcal X_\alpha:=\bigl\{u\in L^2(0,1):{}&
F:=I_0^{1-\alpha}u\in H^1(0,1),\\
&F'-c\in\operatorname{ran}I_1^\alpha
\text{ for some }c\in\C\bigr\}.
\end{aligned}
\end{equation}
If $F'-c=I_1^\alpha f$, then $c$ and $f$ are unique and
\begin{equation}\label{power-natural-representation}
 u=I_0^\alpha I_1^\alpha f+F(0)\phi_0+c\phi_1.
\end{equation}
In this representation, $F'$ has a continuous representative,
$c=F'(1)$, and $f=\Lmax u$.
\end{thm}

\begin{proof}
Suppose first that
$u=Tf+C_0\phi_0+C_1\phi_1$.  The semigroup property gives
\[
 F=I_0^{1-\alpha}u=C_0+C_1x+I_0^1I_1^\alpha f,
 \qquad
 F'=C_1+I_1^\alpha f.
\]
Lemma \ref{weak-right-caputo-inversion} shows that
$F'\in\Dom(\mathcal C_{1,\max}^\alpha)$ and
$\mathcal C_{1,\max}^\alpha F'=f$.  Thus $u\in\Dmax$ and
\eqref{power-natural-action} holds.

Conversely, let $u\in\Dmax$.  Lemma
\ref{weak-right-caputo-inversion} gives unique $C_1\in\C$ and
$f\in L^2(0,1)$ such that
\[
 F'=C_1+I_1^\alpha f,
 \qquad \Lmax u=f.
\]
Writing $C_0=F(0)$ and integrating $F'$ now gives
\[
 F(x)=C_0+C_1x+I_0^1I_1^\alpha f(x).
\]
This is $I_0^{1-\alpha}\widetilde u$, where
$\widetilde u=Tf+C_0\phi_0+C_1\phi_1$.  Injectivity of
$I_0^{1-\alpha}$, proved in the same way, yields $u=\widetilde u$.

The coefficients are unique: $C_0=F(0)$, while $C_1$ and $f$ are uniquely
determined by Lemma \ref{weak-right-caputo-inversion}.  This proves
\eqref{Dmax-def} and
\eqref{power-natural-action}.  The representation immediately implies
$\Dmax\subset\mathcal X_\alpha$.  Conversely, if
$F'-c=I_1^\alpha f$, then Lemma
\ref{weak-right-caputo-inversion} gives
$F'\in\Dom(\mathcal C_{1,\max}^\alpha)$, so
$u\in\Dmax$.
This proves \eqref{power-natural-domain} and
\eqref{power-natural-representation}.
\end{proof}

By Theorem \ref{power-natural-domain-characterization} and Lemma
\ref{abstract-graph-lemma}, $\Lmax$ is closed and
\begin{equation}\label{power-graph-norm-equivalence}
 \|u\|_{\Dmax}^2:=\|u\|_2^2+\|\Lmax u\|_2^2
 \asymp \|f\|_2^2+|C_0|^2+|C_1|^2.
\end{equation}

For $u=Tf+C_0\phi_0+C_1\phi_1\in\Dmax$ define
\begin{equation}\label{traces-minus}
        \xi_1^-(u):=C_0,
        \qquad
        \xi_2^-(u):=C_0+C_1+B_f,
\end{equation}
and
\begin{equation}\label{traces-plus}
        \xi_1^+(u):=C_1+A_f,
        \qquad
        \xi_2^+(u):=C_1.
\end{equation}
Write
\begin{equation}\label{Xi-def}
        \Xiop u:=\bigl(\xi_1^-(u),\xi_2^-(u),
        \xi_1^+(u),\xi_2^+(u)\bigr)^T\in\C^4.
\end{equation}
By \eqref{power-graph-norm-equivalence}, $\Xiop:(\Dmax,\|\cdot\|_{\Dmax})\to\C^4$ is bounded.

These coordinates are canonical endpoint traces.  More precisely, define
\begin{equation}\label{fractional-canonical-trace-representatives}
\begin{split}
        U^-(x)&:=C_0+C_1x+\int_0^x(I_1^\alpha f)(r)\,dr,\\
        U^+(x)&:=C_1+(I_1^\alpha f)(x).
\end{split}
\end{equation}
Then $U^-=I_0^{1-\alpha}u$ and $U^+=D_0^\alpha u$ in the distributional sense.
By the endpoint observation above, $I_1^\alpha f$ is continuous on $[0,1]$ and vanishes at $x=1$.  Hence
\begin{equation}\label{fractional-traces-as-endpoint-values}
        \xi_1^-(u)=U^-(0),\quad
        \xi_2^-(u)=U^-(1),\quad
        \xi_1^+(u)=U^+(0),\quad
        \xi_2^+(u)=U^+(1).
\end{equation}
In particular, the notation $I_0^{1-\alpha}u(0)$ below means the canonical trace limit $U^-(0)$ and not a zero-length integral.

\begin{lemma}\label{trace-surjectivity}
The trace map $\Xiop:\Dmax\to\C^4$ is onto.
\end{lemma}

\begin{proof}
Given $(a,b,c,d)\in\C^4$, choose $C_0=a$ and $C_1=d$.  It remains to choose $f$ such that
\[
        A_f=c-d,
        \qquad
        B_f=b-a-d.
\]
The Riesz representers of $A_f$ and $B_f$ are constant multiples of
$s^{\alpha-1}$ and $s^\alpha$, which are linearly independent in $L^2(0,1)$.
Thus $f\mapsto(A_f,B_f)$ is onto $\C^2$.
\end{proof}

\begin{lemma}[Green formula]\label{green-formula}
For $u,v\in\Dmax$,
\begin{align}\label{eq01-corrected}
(\Lmax u,v)-(u,\Lmax v)
={}&\xi_2^-(u)\overline{\xi_2^+(v)}
      -\xi_2^+(u)\overline{\xi_2^-(v)}       \notag\\
&-\xi_1^-(u)\overline{\xi_1^+(v)}
      +\xi_1^+(u)\overline{\xi_1^-(v)}.
\end{align}
Equivalently, in terms of the canonical representatives,
\begin{equation}\label{green-bracket-form}
(\Lmax u,v)-(u,\Lmax v)
=
\left[
(I_0^{1-\alpha}u)\overline{D_0^\alpha v}
-(D_0^\alpha u)\overline{I_0^{1-\alpha}v}
\right]_{0}^{1}.
\end{equation}
\end{lemma}

\begin{proof}
Write $u=Tf+C_0\phi_0+C_1\phi_1$ and
$v=Tg+E_0\phi_0+E_1\phi_1$.  By \eqref{IBP-integrals},
$T=I_0^\alpha I_1^\alpha=(I_1^\alpha)^*I_1^\alpha$ is self-adjoint.  Therefore the $T$ terms cancel and
\[
(\Lmax u,v)-(u,\Lmax v)
=A_f\overline{E_0}+B_f\overline{E_1}
-C_0\overline{A_g}-C_1\overline{B_g}.
\]
Substitution of \eqref{traces-minus}--\eqref{traces-plus} gives
\eqref{eq01-corrected}; formula \eqref{green-bracket-form} follows from
\eqref{fractional-traces-as-endpoint-values}.
\end{proof}

Define the boundary-minimal, or zero-trace, operator by
\begin{equation}\label{minimal-operator}
        \Dom(\Lmin):=\{u\in\Dmax:\Xiop u=0\},
        \qquad
        \Lmin u:=\Lmax u.
\end{equation}
Because $\Xiop$ is graph-norm bounded, $\Lmin$ is a closed restriction of the closed operator $\Lmax$.  The Green formula shows that it is symmetric.  Here
``boundary-minimal'' means the zero-trace restriction of the graph operator;
no assertion is made that this domain is the closure of a separately chosen
$C_c^\infty(0,1)$ preminimal domain.

\begin{lemma}\label{minimal-adjoint-lemma}
The operator $\Lmin$ is densely defined and
\begin{equation}\label{adjoint-relation}
        \Lmin^*=\Lmax.
\end{equation}
\end{lemma}

\begin{proof}
Let
\[
        N:=\{f\in L^2(0,1):A_f=0,\ B_f=0\}.
\]
The trace formulas give
\[
        \Dom(\Lmin)=\{Tf:f\in N\},
        \qquad \Lmin(Tf)=f.
\]
If $q\perp\Dom(\Lmin)$, then $(Tf,q)=0$ for all $f\in N$.  Since $T=T^*$ and the inner product is linear in the first argument, $(f,Tq)=0$ for all $f\in N$, and hence
$Tq\in N^\perp=\operatorname{span}\{\phi_0,\phi_1\}$.  Write
$Tq=a\phi_0+b\phi_1$.  Since
\[
 \Xiop(Tq)=(0,B_q,A_q,0)^T,
 \qquad
 \Xiop(a\phi_0+b\phi_1)=(a,a+b,b,b)^T,
\]
one obtains $a=b=0$, and hence $Tq=0$.  The injectivity used here follows explicitly from
\[
        (Tq,q)=\|I_1^\alpha q\|_2^2.
\]
Thus $Tq=0$ implies $I_1^\alpha q=0$, and then
$q=D_1^\alpha I_1^\alpha q=0$.  Therefore $\Dom(\Lmin)$ is dense.

The Green formula gives $\Lmax\subset\Lmin^*$.  Conversely, if
$w\in\Dom(\Lmin^*)$, there is $h\in L^2(0,1)$ such that
\[
        (f,w)=(Tf,h),\qquad f\in N.
\]
Since $T=T^*$ and the inner product is linear in the first argument, $(Tf,h)=(f,Th)$.  Thus $(f,w-Th)=0$ for all $f\in N$, so
$w-Th\in N^\perp=\operatorname{span}\{\phi_0,\phi_1\}$.  Hence
$w\in\Dmax$ and $\Lmax w=h$.
\end{proof}

Lemma \ref{minimal-adjoint-lemma} justifies the notation $\Lmax$: it is the
maximal operator in the extension-theoretic sense associated with the
boundary-minimal pair $\Lmin\subset\Lmax$.

For later use set
\begin{equation}\label{Gamma-fractional-definition}
        \Gzero u=(\xi_1^-(u),\xi_2^-(u))^T,
        \qquad
        \Gone u=(\xi_1^+(u),-\xi_2^+(u))^T.
\end{equation}
Then \eqref{eq01-corrected} takes the standard boundary-triplet form
\begin{equation}\label{fractional-boundary-triplet-form}
(\Lmax u,v)-(u,\Lmax v)
=(\Gone u,\Gzero v)_{\C^2}-(\Gzero u,\Gone v)_{\C^2}.
\end{equation}

\begin{prop}\label{power-boundary-triplet}
The triple $(\C^2,\Gzero,\Gone)$ is an ordinary boundary triplet for $\Lmin^*=\Lmax$ in the standard sense of \cite{GG91,BHdS20}.
In particular, the deficiency indices of $\Lmin$ are $(2,2)$.
\end{prop}

\begin{proof}
The Green identity is \eqref{fractional-boundary-triplet-form}.  By Lemma
\ref{trace-surjectivity}, the map
$u\mapsto(\Gzero u,\Gone u)$ from $\Dmax$ onto $\C^2\oplus\C^2$ is surjective.
Together with Lemma \ref{minimal-adjoint-lemma}, these are precisely the defining properties of an ordinary boundary triplet; see \cite{GG91,BHdS20}.
\end{proof}

\subsection{Self-adjoint boundary conditions}\label{sec:power-selfadjoint}

The Green form is represented by
\begin{equation}\label{J-def}
J=
\begin{pmatrix}
0&0&-1&0\\
0&0&0&1\\
1&0&0&0\\
0&-1&0&0
\end{pmatrix},
\qquad J^*=-J,
\qquad J^2=-I_4,
\end{equation}
so that
\begin{equation}\label{green-J-form}
        (\Lmax u,v)-(u,\Lmax v)=(\Xiop u,J\Xiop v)_{\C^4}.
\end{equation}

\begin{defi}\label{admissible-boundary-matrix}
A matrix $\omega\in\C^{2\times4}$ is called admissible if
\begin{equation}\label{admissible-conditions}
        \rank\omega=2,
        \qquad
        \omega J\omega^*=0.
\end{equation}
Two admissible matrices that differ by left multiplication with an element of
$\mathrm{GL}(2,\C)$ define the same boundary subspace.
\end{defi}

For an admissible $\omega$, define
\begin{equation}\label{Lomega-domain}
\Dom(\mathcal L_\omega)=\{u\in\Dmax:\omega\Xiop u=0\},
        \qquad
        \mathcal L_\omega u=\Lmax u.
\end{equation}

\begin{thm}\label{power-all-selfadjoint-extensions}
For every admissible boundary matrix $\omega$, the operator $\mathcal L_\omega$ is a self-adjoint extension of $\Lmin$.  Conversely, every self-adjoint extension of $\Lmin$ is obtained in this way.
\end{thm}

\begin{proof}
By Proposition \ref{power-boundary-triplet} and the standard boundary-triplet correspondence \cite{BHdS20,GG91}, self-adjoint extensions of $\Lmin$ correspond bijectively to Lagrangian subspaces of the boundary space $\C^4$ equipped with the form $(z,Jw)_{\C^4}$.  The subspace
$\ker\omega$ has dimension two because $\rank\omega=2$.  With the convention
\[
\mathcal M^{\perp_J}:=\{z\in\C^4:(w,Jz)_{\C^4}=0
        \text{ for all }w\in\mathcal M\},
\]
the identity $J^2=-I$ gives
$(\ker\omega)^{\perp_J}=-J\operatorname{ran}\omega^*$.
Equivalently this is the same subspace as $J\operatorname{ran}\omega^*$, since multiplication by $-1$ does not change a linear subspace.  The identity $\omega J\omega^*=0$ says that
$J\operatorname{ran}\omega^*\subset\ker\omega$; both spaces have dimension two, so they are equal.  Thus $\ker\omega$ equals its boundary-form orthogonal complement and is Lagrangian.  Conversely, if $\mathcal M$ is a Lagrangian plane and $\omega$ is any rank-two matrix with $\ker\omega=\mathcal M$, then
$J\operatorname{ran}\omega^*=\mathcal M$, which implies
$\omega J\omega^*=0$.  Surjectivity of $\Xiop$ transfers this finite-dimensional parametrization to the operator domains.
\end{proof}

\begin{rem}
Equivalently, every self-adjoint extension can be written in boundary-triplet coordinates as
\[
        A\Gzero u+B\Gone u=0,
        \qquad
        \rank(A\;B)=2,
        \qquad
        AB^*=BA^*.
\]
This form is often convenient when comparing the present construction with standard boundary-triplet and extension theory \cite{BHdS20,GG91}.
\end{rem}

\begin{rem}[Finite-rank resolvent differences]\label{finite-rank-resolvent-differences}
The ordinary boundary-triplet correspondence also gives the usual Krein-type finite-dimensional resolvent correction \cite{BHdS20,GG91}.  If $\mathcal L_{\omega_1}$ and $\mathcal L_{\omega_2}$ are two self-adjoint extensions and
\[
        z\in\rho(\mathcal L_{\omega_1})\cap
        \rho(\mathcal L_{\omega_2}),
\]
then
\[
        \rank\left((\mathcal L_{\omega_1}-z)^{-1}
        -(\mathcal L_{\omega_2}-z)^{-1}\right)\le2.
\]
Formula \eqref{inverse-general-formula} below is the explicit $z=0$ version of this finite-dimensional correction.
\end{rem}

\subsection{Zero modes, inverse formulas and compactness}\label{sec:power-inverse}

The trace vectors of the homogeneous modes and of the particular solution are
\begin{equation}\label{homogeneous-traces}
\Xiop\phi_0=(1,1,0,0)^T,
        \qquad
\Xiop\phi_1=(0,1,1,1)^T,
\end{equation}
and
\begin{equation}\label{Xi-Tf}
        \Xiop(Tf)=(0,B_f,A_f,0)^T.
\end{equation}
For $\omega=(\omega_{ij})\in\C^{2\times4}$ define
\begin{equation}\label{Momega-def}
M_\omega:=
\begin{pmatrix}
\omega_{11}+\omega_{12} & \omega_{12}+\omega_{13}+\omega_{14}\\
\omega_{21}+\omega_{22} & \omega_{22}+\omega_{23}+\omega_{24}
\end{pmatrix}
\end{equation}
and
\begin{equation}\label{Womega-def}
W_\omega:=
\begin{pmatrix}
\omega_{12}&\omega_{13}\\
\omega_{22}&\omega_{23}
\end{pmatrix}.
\end{equation}

\begin{rem}[Invariance under a change of boundary matrix]\label{boundary-matrix-invariance}
If $\widetilde\omega=G\omega$ with $G\in\mathrm{GL}(2,\C)$, then
$\ker\widetilde\omega=\ker\omega$ and
\[
        M_{\widetilde\omega}=GM_\omega,
        \qquad
        W_{\widetilde\omega}=GW_\omega.
\]
Consequently,
\[
 \det M_{\widetilde\omega}\ne0
 \quad\Longleftrightarrow\quad
 \det M_\omega\ne0,
 \qquad
 \rank W_{\widetilde\omega}=\rank W_\omega.
\]
Whenever $M_\omega$ is invertible,
\[
        M_{\widetilde\omega}^{-1}\widetilde\omega
        =(GM_\omega)^{-1}G\omega
        =M_\omega^{-1}\omega.
\]
Thus the zero criterion, inverse formula and exact correction rank below depend only on the boundary subspace, not on its chosen matrix representative.
\end{rem}

\begin{prop}[Zero-eigenvalue criterion]\label{kernel-criterion}
For an admissible $\omega$,
\begin{equation}\label{kernel-formula}
\ker\mathcal L_\omega
=\{C_0\phi_0+C_1\phi_1:M_\omega(C_0,C_1)^T=0\}.
\end{equation}
In particular,
\[
        0\in\rho(\mathcal L_\omega)
        \quad\Longleftrightarrow\quad
        \det M_\omega\ne0.
\]
\end{prop}

\begin{proof}
If $u\in\ker\mathcal L_\omega$, then the canonical representation of $u$
has $f=0$, so $u=C_0\phi_0+C_1\phi_1$.  Substitution of
\eqref{homogeneous-traces} into $\omega\Xiop u=0$ gives
\eqref{kernel-formula}.  If $\det M_\omega=0$, a nonzero vector in
$\ker M_\omega$ produces a nonzero element of $\ker\mathcal L_\omega$, so
$0\notin\rho(\mathcal L_\omega)$.  If $\det M_\omega\ne0$, then for every
$f\in L^2(0,1)$ the boundary system for $(C_0,C_1)$ has a unique solution,
and the resulting solution operator is bounded because $T,A_f,B_f$ are
bounded and $\phi_0,\phi_1\in L^2(0,1)$.  Hence
$0\in\rho(\mathcal L_\omega)$.  The resulting inverse is written explicitly
in Theorem~\ref{inverse-general}.
\end{proof}

\begin{thm}[Explicit inverse]\label{inverse-general}
Let $\omega$ be admissible and assume $\det M_\omega\ne0$.  Then
\begin{equation}\label{inverse-general-formula}
\mathcal L_\omega^{-1}f
=Tf-
\begin{pmatrix}\phi_0&\phi_1\end{pmatrix}
M_\omega^{-1}\omega
\begin{pmatrix}
0\\ B_f\\ A_f\\ 0
\end{pmatrix},
        \qquad f\in L^2(0,1).
\end{equation}
Thus $\mathcal L_\omega^{-1}$ is $I_0^\alpha I_1^\alpha$ plus an operator of rank at most two.
\end{thm}

\begin{proof}
Every solution of $\Lmax u=f$ has the form
$u=Tf+C_0\phi_0+C_1\phi_1$.  By \eqref{homogeneous-traces} and
\eqref{Xi-Tf}, the boundary condition is equivalent to
\[
M_\omega\binom{C_0}{C_1}
=-\omega\begin{pmatrix}0\\B_f\\A_f\\0\end{pmatrix}.
\]
The assumed invertibility of $M_\omega$ gives the unique solution
\eqref{inverse-general-formula}.  Boundedness follows from the boundedness of
$T,A_f,B_f$ and the inclusion $\phi_0,\phi_1\in L^2(0,1)$.
\end{proof}

\begin{prop}[Exact rank of the inverse correction]\label{exact-rank-correction}
Let $\omega$ be admissible and assume $\det M_\omega\ne0$.  Then
\[
        \rank(\mathcal L_\omega^{-1}-T)=\rank W_\omega .
\]
In particular, the rank-two bound in Theorem \ref{inverse-general} is sharp.
\end{prop}

\begin{proof}
From \eqref{inverse-general-formula},
\[
(\mathcal L_\omega^{-1}-T)f
=
-\begin{pmatrix}\phi_0&\phi_1\end{pmatrix}
M_\omega^{-1}W_\omega
\binom{B_f}{A_f}.
\]
The map $f\mapsto(B_f,A_f)$ is onto $\C^2$, because its Riesz representers are constant multiples of $s^\alpha$ and $s^{\alpha-1}$, which are linearly independent in $L^2(0,1)$.  The matrix $M_\omega^{-1}$ is invertible and the map $(c_0,c_1)\mapsto c_0\phi_0+c_1\phi_1$ is injective.  Therefore the rank is exactly $\rank W_\omega$.
\end{proof}

For the reference extension $\omega_+$ below the correction has rank zero.  For
\[
        \omega_D=
        \begin{pmatrix}1&0&0&0\\0&1&0&0\end{pmatrix}
\]
it has rank one, and for $\omega_0$ in Example \ref{reference-extension-example} it has rank two.  Thus all possible ranks occur among invertible self-adjoint realizations.

\begin{ex}[A positive reference extension]\label{positive-reference-power}
Let
\begin{equation}\label{omega-positive-power}
        \omega_+=
        \begin{pmatrix}1&0&0&0\\0&0&0&1\end{pmatrix}.
\end{equation}
The boundary conditions are
\[
        \xi_1^-(u)=0,
        \qquad
        \xi_2^+(u)=0.
\]
They force $C_0=C_1=0$, so
\begin{equation}\label{positive-reference-inverse-power}
        \mathcal L_{\omega_+}^{-1}=T=(I_1^\alpha)^*I_1^\alpha.
\end{equation}
Indeed, every $u\in\Dom(\mathcal L_{\omega_+})$ has the form $u=Tf$, and
\[
 (\mathcal L_{\omega_+}u,u)=(f,Tf)=\|I_1^\alpha f\|_2^2.
\]
Since $T$ is positive and $0\le T^2\le\|T\|T$,
\[
 \|u\|_2^2=\|Tf\|_2^2=(f,T^2f)
 \le \|T\|(f,Tf).
\]
Therefore
\begin{equation}\label{positive-reference-lower-bound}
        \mathcal L_{\omega_+}\ge \|T\|^{-1}I,
\end{equation}
so the reference realization is strictly positive.  Moreover,
$I_1^\alpha$ is Hilbert--Schmidt for $\alpha>1/2$, and hence
$T=(I_1^\alpha)^*I_1^\alpha$ is compact; thus the inverse is compact.
\end{ex}

\begin{ex}[Homogeneous corrections are generally necessary]\label{reference-extension-example}
Consider
\begin{equation}\label{BC-0-corrected}
        \xi_2^-(u)=0,
        \qquad
        \xi_1^+(u)=0,
\end{equation}
for which
\[
        \omega_0=
        \begin{pmatrix}0&1&0&0\\0&0&1&0\end{pmatrix},
        \qquad
        M_{\omega_0}=\begin{pmatrix}1&1\\0&1\end{pmatrix}.
\]
Formula \eqref{inverse-general-formula} becomes
\begin{equation}\label{L0-inverse-corrected}
\mathcal L_{\omega_0}^{-1}f
=I_0^\alpha I_1^\alpha f
+(A_f-B_f)\frac{x^{\alpha-1}}{\Gamma(\alpha)}
-A_f\frac{x^\alpha}{\Gamma(\alpha+1)}.
\end{equation}
No moment restriction is imposed on $f$.
\end{ex}

\begin{cor}[Compact resolvent for every self-adjoint extension]\label{compact-resolvent-corollary}
Every self-adjoint extension $\mathcal L_\omega$ has compact resolvent.  Hence its spectrum is real and purely discrete, with finite-multiplicity eigenvalues having no finite accumulation point, and its eigenfunctions form an orthonormal basis of $L^2(0,1)$.  If $0\in\rho(\mathcal L_\omega)$, then $\mathcal L_\omega^{-1}$ is compact.
\end{cor}

\begin{proof}
Since $\alpha>1/2$, the kernel
$(s-x)^{\alpha-1}\mathbf 1_{x<s}/\Gamma(\alpha)$ of $I_1^\alpha$
belongs to $L^2((0,1)^2)$.  Thus $I_1^\alpha$ is Hilbert--Schmidt and
$T=(I_1^\alpha)^*I_1^\alpha$ is compact.  For the reference extension,
$\mathcal L_{\omega_+}^{-1}=T$.  Since $T$ is self-adjoint,
$I-iT$ is boundedly invertible, and
\[
 (\mathcal L_{\omega_+}-i)^{-1}
 =(I-iT)^{-1}T=T(I-iT)^{-1}
\]
is compact.  The point $i$ belongs to the resolvent set of every self-adjoint extension.  By Remark \ref{finite-rank-resolvent-differences},
\[
 (\mathcal L_\omega-i)^{-1}
 -(\mathcal L_{\omega_+}-i)^{-1}
\]
has rank at most two.  Hence $(\mathcal L_\omega-i)^{-1}$ is compact for every admissible $\omega$.  The spectral conclusions follow from the compact-resolvent spectral theorem.  If $0\in\rho(\mathcal L_\omega)$, compactness of the inverse follows by taking the resolvent at zero.
\end{proof}

\subsection{Examples and the zero-eigenvalue obstruction}\label{sec:power-examples}

The following matrices are admissible; in the family $\omega_\rho$ one takes
$\rho\in\mathbb R$:
\begin{equation}\label{MatrixExamplesCorrected}
\omega_D=\begin{pmatrix}1&0&0&0\\0&1&0&0\end{pmatrix},
\qquad
\omega_+=\begin{pmatrix}1&0&0&0\\0&0&0&1\end{pmatrix},
\end{equation}
\begin{equation}\label{MatrixExamplesCorrected2}
\omega_\rho=\begin{pmatrix}\rho&1&0&0\\0&0&1&\rho\end{pmatrix},
\qquad
\omega_N=\begin{pmatrix}0&0&1&0\\0&0&0&1\end{pmatrix}.
\end{equation}
Admissibility does not imply invertibility at zero.  Direct calculation gives
\begin{equation}\label{det-rho}
        \det M_{\omega_\rho}=(\rho+1)^2.
\end{equation}
Thus $\mathcal L_{\omega_\rho}^{-1}$ exists for $\rho\ne-1$, whereas at
$\rho=-1$ one has
\[
        \ker\mathcal L_{\omega_{-1}}
        =\operatorname{span}\{\phi_0\}.
\]
Likewise,
\begin{equation}\label{neumann-kernel}
        \ker\mathcal L_{\omega_N}
        =\operatorname{span}\left\{
        \frac{x^{\alpha-1}}{\Gamma(\alpha)}\right\}.
\end{equation}
The operator $\mathcal L_{\omega_N}$ is self-adjoint but has no everywhere-defined inverse.  By Corollary \ref{compact-resolvent-corollary}, zero is an isolated eigenvalue of finite multiplicity, and the reduced inverse on
$(\ker\mathcal L_{\omega_N})^\perp$ is compact.

\section{Sharp singular-value asymptotics}\label{sec:power-schatten}

\begin{thm}[Sharp singular-value and Weyl asymptotics]\label{power-sharp-schatten}
Let $\omega$ be admissible and assume $0\in\rho(\mathcal L_\omega)$.  Then
\begin{equation}\label{sharp-leading-singular-values}
        s_n(\mathcal L_\omega^{-1})
        =(\pi n)^{-2\alpha}\bigl(1+O(n^{-1})\bigr),
        \qquad n\to\infty.
\end{equation}
Consequently,
\begin{equation}\label{Schatten-sharp}
        \mathcal L_\omega^{-1}\in S_p(L^2(0,1))
        \quad\Longleftrightarrow\quad
        p>\frac1{2\alpha}.
\end{equation}
At the critical exponent $p_0=1/(2\alpha)$ one has
\begin{equation}\label{weak-critical-schatten}
        \mathcal L_\omega^{-1}\in S_{p_0,\infty}
        \setminus S_{p_0},
        \qquad
        \lim_{n\to\infty} n^{2\alpha}s_n(\mathcal L_\omega^{-1})
        =\pi^{-2\alpha},
\end{equation}
and
\begin{equation}\label{critical-log-sum}
        \sum_{n\le N}s_n(\mathcal L_\omega^{-1})^{p_0}
        \sim \frac1\pi\log N.
\end{equation}
If $\nu_n$ denotes the eigenvalues of $|\mathcal L_\omega|$ arranged in nondecreasing order and counted with multiplicities, then
\begin{equation}\label{absolute-eigenvalue-asymptotics}
        \nu_n=(\pi n)^{2\alpha}\bigl(1+O(n^{-1})\bigr).
\end{equation}
Equivalently, for the counting function of the absolute spectrum,
\begin{equation}\label{weyl-counting-asymptotics}
        N_\omega(R):=\#\{n\in\mathbb N:\nu_n\le R\}
        =\frac1\pi R^{1/(2\alpha)}+O(1),
        \qquad R\to\infty.
\end{equation}
\end{thm}

\begin{proof}
Let $K_\alpha=I_0^\alpha$ denote the left Riemann--Liouville fractional
integration operator on $L^2(0,1)$.  Because the present paper assumes
$\alpha>1/2$, we invoke directly Burman's sharp estimate \cite{Bur07},
\begin{equation}\label{left-RL-singular-asymptotics}
        s_n(K_\alpha)=(\pi n)^{-\alpha}\bigl(1+O(n^{-1})\bigr),
        \qquad n\to\infty.
\end{equation}
The right-sided operator $I_1^\alpha$ is unitarily equivalent to $K_\alpha$ by the reflection $(Rf)(x)=f(1-x)$.  Therefore
\begin{equation}\label{RL-singular-asymptotics}
        s_n(I_1^\alpha)=(\pi n)^{-\alpha}\bigl(1+O(n^{-1})\bigr).
\end{equation}
Since $T=(I_1^\alpha)^*I_1^\alpha$, the nonzero singular values of $T$ are the squares of the singular values of $I_1^\alpha$.  Hence
\begin{equation}\label{T-sharp-singular-values}
        s_n(T)=(\pi n)^{-2\alpha}\bigl(1+O(n^{-1})\bigr).
\end{equation}
By \eqref{inverse-general-formula},
$F_\omega:=\mathcal L_\omega^{-1}-T$ has finite rank
$r=\rank F_\omega\le2$.  The standard singular-value inequalities for
compact operators \cite{GK69,Sim05} imply, for $n>r$,
\begin{equation}\label{finite-rank-singular-interlacing}
        s_{n+r}(T)
        \le s_n(T+F_\omega)
        \le s_{n-r}(T),
\end{equation}
with the evident interpretation when $r=0$.
Combining \eqref{T-sharp-singular-values} and \eqref{finite-rank-singular-interlacing} gives \eqref{sharp-leading-singular-values}.

The Schatten criterion follows because
$\sum n^{-2\alpha p}$ converges exactly when $2\alpha p>1$.  At
$p_0=1/(2\alpha)$, \eqref{sharp-leading-singular-values} gives
\[
        s_n(\mathcal L_\omega^{-1})^{p_0}
        =(\pi n)^{-1}\bigl(1+O(n^{-1})\bigr),
\]
which proves both the weak-Schatten membership and the logarithmic partial-sum formula \eqref{critical-log-sum}.  Since $\mathcal L_\omega$ is
self-adjoint and invertible, the functional calculus gives
\[
        |\mathcal L_\omega^{-1}|=|\mathcal L_\omega|^{-1}.
\]
Therefore, with $\nu_n$ arranged in nondecreasing order and singular
values arranged in nonincreasing order,
\[
        s_n(\mathcal L_\omega^{-1})=\nu_n^{-1}.
\]
Inverting \eqref{sharp-leading-singular-values} gives
\eqref{absolute-eigenvalue-asymptotics}.  Taking $2\alpha$-th roots yields
\[
        \nu_n^{1/(2\alpha)}=\pi n+O(1).
\]
Hence there is $C>0$ such that, for all sufficiently large $n$,
\[
        \pi n-C\le \nu_n^{1/(2\alpha)}\le \pi n+C.
\]
These two inequalities bound $N_\omega(R)$ between
$\pi^{-1}R^{1/(2\alpha)}+O(1)$ from below and above, proving
\eqref{weyl-counting-asymptotics}.
\end{proof}

\begin{rem}[Reference phase asymptotics and scope]\label{reference-phase-comparison}
For the positive reference realization,
$\mathcal L_{\omega_+}^{-1}=T=I_0^\alpha I_1^\alpha$.  Its integral kernel is
\[
 K_\alpha(x,y)=\frac1{\Gamma(\alpha)^2}
 \int_0^{x\wedge y}(x-t)^{\alpha-1}(y-t)^{\alpha-1}\,dt,
\]
which is the kernel of problem $(P')$ in \cite[pp.~719--720]{CK21}.
For that reference problem, the sharper frequency asymptotic
\[
 \lambda_n^{1/(2\alpha)}
 =\pi n-\frac{\pi}{2}+O(n^{-1})
\]
is already known.  Theorem \ref{power-sharp-schatten} does not improve this
phase formula.  Its contribution is the propagation of the common leading
singular-value constant, optimal Schatten threshold, weak endpoint and Weyl
constant to every invertible self-adjoint boundary plane.  A finite-rank
argument alone does not determine the boundary-dependent second term.
\end{rem}

\begin{rem}[Sharpness]
Let $r=\rank F_\omega$.  The finite-rank singular-value inequalities give
\[
        s_{n+r}(T)\le s_n(T+F_\omega)\le s_{n-r}(T),
        \qquad n>r.
\]
Thus the boundary correction changes the singular-value sequence by at most a bounded index shift and cannot alter the leading constant in
\eqref{sharp-leading-singular-values}.  Consequently the Schatten threshold, the weak endpoint and the Weyl constant are independent of the admissible self-adjoint boundary condition, as long as zero is not an eigenvalue.
\end{rem}

\raggedbottom
\section{Relation to the earlier power-kernel papers and improvement record}
\label{sec:comparison-previous}

The expression considered here, and its unitary reflection, occur in
\cite{TT16,TT18G,TT18M,TT19}.  The ordering
$\mathcal D_1^\alpha D_0^\alpha$ is used in \cite{TT18G,TT19}; the reflected
left-Caputo/right-Riemann--Liouville ordering is used in
\cite{TT16,TT18M}.  Comparisons with the latter are understood after reflection
of the interval.  Those papers use H\"older-type spaces or graph closures,
whereas the present domain is independently defined by
\eqref{distributional-graph-domain} and represented by \eqref{Dmax-def}.  The
respective domains are not identified.  Table
\ref{tab:precise-correction-record} records the earlier formulas or
consequences that are replaced within the present $L^2$ graph realization.

\begin{equation}\label{maximal-kernel-exact}
        \ker\Lmax=\operatorname{span}\{\phi_0,\phi_1\}.
\end{equation}
The sign correction follows directly from the endpoint bracket
\[
 \bigl[(I_0^{1-\alpha}u)\overline{D_0^\alpha v}
 -(D_0^\alpha u)\overline{I_0^{1-\alpha}v}\bigr]_0^1,
\]
whose expansion necessarily gives opposite signs at the two endpoints.  The
particular solution \(Tf=I_0^\alpha I_1^\alpha f\) has trace vector
\(\Xiop(Tf)=(0,B_f,A_f,0)^T\), so the moment powers are \(\alpha\) and
\(\alpha-1\), not \(2\alpha\) and \(2\alpha-1\).  The exact inverse formula
\eqref{inverse-general-formula} then supplies the boundary-dependent homogeneous correction and simultaneously exposes the zero-eigenvalue obstruction through \(M_\omega\).

Burman's sharp singular-value asymptotics for fractional integration
\cite{Bur07}, combined with the finite-rank inverse correction, give
\[
 s_n(\mathcal L_\omega^{-1})
   =(\pi n)^{-2\alpha}\bigl(1+O(n^{-1})\bigr),
\]
for every admissible realization satisfying \(0\in\rho(\mathcal L_\omega)\).
Thus the sharp Schatten condition is \(p>1/(2\alpha)\), rather than the condition
stated in \cite[Corollary~3.6]{TT18G} and repeated in
\cite[Corollary~3.6]{TT19}.  Invertibility must also be checked: the boundary
matrices $\omega_N$ and $\omega_{-1}$ used below possess the zero mode $\phi_0$.

For the represented graph domain, the graph, extension and spectral sections
replace the corresponding homogeneous-mode, Green, inverse, zero-mode and
Schatten formulas in
\cite{TT16,TT18G,TT18M,TT19}.  Theorem
\ref{power-natural-domain-characterization} proves the Volterra representation
from the endpoint-normalized distributional definition, and Lemma
\ref{minimal-adjoint-lemma} identifies the resulting graph operator as the
adjoint of its boundary-minimal zero-trace restriction.

Compact integral formulations, exact and numerical solutions, and discrete
spectral conclusions for selected fractional boundary conditions were already
developed in \cite{KCB18}.  Homogeneous Robin conditions and discreteness are
treated in \cite{Kli19}, while the integral Hilbert--Schmidt approach in
\cite{Kli21} proves discreteness and basis properties for homogeneous Dirichlet
fractional and fractional Prabhakar problems.  Moreover, the positive reference
inverse and a sharper two-term reference asymptotic occur in \cite{CK21}; see
Remark \ref{reference-phase-comparison}.  Thus compactness, discreteness and the
reference inverse are not distinguishing points of the present treatment.  The
distinguishing feature is their combination with a closed $L^2$ graph domain,
an ordinary boundary triplet for all coupled and separated self-adjoint planes,
exact zero-mode and inverse matrices, exact correction rank, and a
boundary-independent leading spectral law.

\begin{table}[H]
\caption{Improvement record for the 2016--2019 power-kernel formulas.}
\label{tab:precise-correction-record}
\small
\setlength{\tabcolsep}{4pt}
\renewcommand{\arraystretch}{1.18}
\begin{tabularx}{\textwidth}{@{}>{\raggedright\arraybackslash}p{0.25\textwidth}YY@{}}
\toprule
Earlier location & Formula or consequence used there & Replacement in the present graph realization \\
\midrule
\cite{TT16}, Lemma~4.1 and the paragraph following equation~(4.4);
\cite[Lemma~1 and p.~183]{TT18M}
& Shifted truncated powers are treated as homogeneous vectors and are used to infer infinitely many linearly independent elements of the maximal kernel.
& The unique representation \eqref{Dmax-def} gives the exact two-dimensional kernel \eqref{maximal-kernel-exact}.  Only the endpoint modes \(\phi_0\) and \(\phi_1\) occur in the present $L^2$ graph domain. \\
\addlinespace
\cite[equation~(4.1)]{TT16}; \cite[equation~(3.6)]{TT18M};
\cite[Lemma~3.1, equation~(3.5)]{TT18G};
\cite[Lemma~3.1, equation~(3.5)]{TT19}
& The displayed boundary sums assign the same orientation to both endpoint pairs.
& The left and right endpoints have opposite orientations, as shown by the Green identity \eqref{eq01-corrected} and the bracket formula \eqref{green-bracket-form}. \\
\addlinespace
\cite[p.~180]{TT18M};
\cite[the paragraph following equation~(3.3) and pp.~475--476]{TT18G};
\cite[the paragraph following equation~(3.3)]{TT19}
& The compatibility moments are written with powers \(2\alpha\) and \(2\alpha-1\), and the inverse is represented by the bare product \(I_0^\alpha I_1^\alpha\) on the resulting constrained subspace.
& The graph traces of \(Tf\) are generated by the moments \(B_f\) and \(A_f\) in \eqref{power-moments}, with powers \(\alpha\) and \(\alpha-1\).  For the boundary conditions \eqref{BC-0-corrected}, the inverse on all of \(L^2(0,1)\) is \eqref{L0-inverse-corrected}, including both homogeneous corrections. \\
\addlinespace
\cite[Lemma~3.5 and equation~(3.8)]{TT18G};
\cite[Lemma~3.5 and equation~(3.8)]{TT19}, together with the corresponding corollaries
& An ordinary inverse and a Schatten conclusion are stated for every listed boundary matrix.
& Invertibility requires \(\det M_\omega\ne0\); see Proposition \ref{kernel-criterion}.  In particular, \(\omega_N\) and the member \(\omega_{-1}\) of the \(\omega_\rho\)-family have the zero mode \(\phi_0\), as shown in \eqref{det-rho}--\eqref{neumann-kernel}. \\
\addlinespace
\cite[Corollary~3.6]{TT18G}; \cite[Corollary~3.6]{TT19}
& The claimed Schatten condition is \(p>2/(1+4\alpha)\).
& The exact asymptotic formula \eqref{sharp-leading-singular-values} gives the sharp equivalence \(\mathcal L_\omega^{-1}\in S_p\) if and only if \(p>1/(2\alpha)\), subject to \(0\in\rho(\mathcal L_\omega)\). \\
\bottomrule
\end{tabularx}
\end{table}
\clearpage
\flushbottom

\section{A Caputo-time diffusion equation}\label{sec:diffusion}

Let $H=L^2(0,1)$, and let $A=\mathcal L_\omega$ be a self-adjoint
realization from Theorem \ref{power-all-selfadjoint-extensions}.  We assume
\begin{equation}\label{diffusion-positive-assumption}
 A\ge a_0I\qquad\text{for some }a_0>0.
\end{equation}
For a general admissible $\omega$, this is an additional hypothesis.  It is
satisfied by the reference realization $A=\mathcal L_{\omega_+}$ by
\eqref{positive-reference-lower-bound}.  Since $A$ has compact resolvent, it
has an orthonormal eigenbasis
\begin{equation}\label{diffusion-eigenpairs}
        Ae_n=\lambda_ne_n,
        \qquad 0<a_0\le\lambda_1\le\lambda_2\le\dots,
        \qquad \lambda_n\to\infty.
\end{equation}
Eigenvalues are repeated according to multiplicity.
For $0<\beta<1$ we consider
\begin{equation}\label{diffusion-abstract-problem}
\begin{cases}
{}^{C}D_{0+,t}^{\beta}u(t)+Au(t)=f(t),&0<t\le T,\\
u(0)=\varphi.
\end{cases}
\end{equation}
Here ${}^{C}D_{0+,t}^{\beta}$ is the Caputo derivative and
\[
 E_{\beta,\gamma}(z):=\sum_{k=0}^{\infty}
 \frac{z^k}{\Gamma(\beta k+\gamma)},
 \qquad E_\beta(z):=E_{\beta,1}(z)
\]
are the Mittag--Leffler functions; see \cite[Eq.~(1.56)]{KST06}.  The
fractional-evolution and eigenfunction-expansion framework is standard; see
\cite[Chapter~2]{Baz01} and \cite[Sections~2--3]{SY11}.  By the Borel
functional calculus set
\begin{equation}\label{diffusion-operator-families}
 S_\beta(t):=E_\beta(-t^\beta A),\quad t\ge0,
 \qquad
 P_\beta(t):=t^{\beta-1}E_{\beta,\beta}(-t^\beta A),\quad t>0.
\end{equation}
At the data regularity used below, problem \eqref{diffusion-abstract-problem}
is understood only in the following mild sense.  In particular, no assertion
is made that $Au$ or ${}^{C}D_{0+,t}^{\beta}u$ is an $H$-valued function.

\begin{defi}[Mild solution]\label{diffusion-mild-definition}
A function $u\in C([0,T];H)$ is a mild solution of
\eqref{diffusion-abstract-problem} if, for every $t\in[0,T]$,
\begin{equation}\label{diffusion-mild-definition-formula}
        u(t)=S_\beta(t)\varphi+
        \int_0^tP_\beta(t-s)f(s)\,ds,
\end{equation}
where the integral is a Bochner integral in $H$.
\end{defi}

\begin{thm}[Existence, uniqueness and stability of the mild solution]\label{diffusion-mild-theorem}
Let $0<\beta<1$, $\varphi\in H$, and $f\in C([0,T];H)$.  Then the
right-hand side of \eqref{diffusion-mild-definition-formula} defines a unique
mild solution $u\in C([0,T];H)$ in the sense of Definition
\ref{diffusion-mild-definition}, and it is given by
\begin{align}\label{diffusion-solution}
 u(t)
 ={}&\sum_{n=1}^{\infty}\varphi_n
       E_{\beta}(-\lambda_nt^{\beta})e_n \notag\\
 &+\sum_{n=1}^{\infty}
   \left[\int_0^t(t-s)^{\beta-1}
   E_{\beta,\beta}\bigl(-\lambda_n(t-s)^{\beta}\bigr)f_n(s)\,ds\right]e_n,
\end{align}
where $\varphi_n=(\varphi,e_n)$ and $f_n(s)=(f(s),e_n)$.  The series converge in $H$ uniformly for $t\in[0,T]$.  Moreover,
\begin{equation}\label{diffusion-stability-estimate}
        \sup_{0\le t\le T}\|u(t)\|_H
        \le C_{\beta,T}
        \left(\|\varphi\|_H+\sup_{0\le t\le T}\|f(t)\|_H\right).
\end{equation}
\end{thm}

\begin{proof}
By the spectral theorem, the operator families in
\eqref{diffusion-operator-families} have the expansions
\[
 S_\beta(t)\varphi
 =\sum_{n=1}^{\infty}E_\beta(-\lambda_n t^\beta)
      \varphi_n e_n
\]
and
\[
 P_\beta(t)g
 =\sum_{n=1}^{\infty}t^{\beta-1}
 E_{\beta,\beta}(-\lambda_n t^\beta)(g,e_n)e_n,
 \qquad t>0.
\]
The negative-axis Mittag--Leffler estimate
\cite[Theorem~1.6]{KST06}, equivalently \cite[Lemma~3.1]{SY11}, gives
\[
 |E_\beta(-r)|\le \frac{C_\beta}{1+r},
 \qquad
 |E_{\beta,\beta}(-r)|\le \frac{C_\beta}{1+r},
 \qquad r\ge0;
\]
Hence $S_\beta(t)$ is uniformly bounded on $[0,T]$, while
\[
        \|P_\beta(t)\|\le C_\beta t^{\beta-1},
        \qquad 0<t\le T.
\]
Consequently the Bochner integral
\[
        v(t):=\int_0^tP_\beta(t-s)f(s)\,ds
\]
is well defined and satisfies
\[
        \|v(t)\|_H
        \le \frac{C_\beta t^\beta}{\beta}
        \sup_{0\le s\le T}\|f(s)\|_H.
\]
Dominated convergence in the spectral expansion shows that
$t\mapsto S_\beta(t)\varphi$ is continuous on $[0,T]$ and that
$S_\beta(t)\varphi\to\varphi$ as $t\downarrow0$.
The displayed estimate gives $v(t)\to0$ as $t\downarrow0$.  To prove continuity at a fixed $t_0>0$, choose $0<\delta<t_0/2$ and split the convolution into the regions $0<t-s<\delta$ and $t-s\ge\delta$.  The first part is bounded uniformly for $t$ near $t_0$ by
\[
        \frac{C_\beta\delta^\beta}{\beta}
        \sup_{0\le s\le T}\|f(s)\|_H.
\]
On the second region, $P_\beta(r)$ is strongly continuous and uniformly
bounded for $r\in[\delta,T]$.  Indeed, if
\[
        p_r(\lambda):=r^{\beta-1}E_{\beta,\beta}(-\lambda r^\beta),
\]
then $p_r(\lambda_n)$ is continuous in $r$ for every $n$ and is uniformly
bounded for $r\in[\delta,T]$ and $n\ge1$; Parseval's identity and dominated
convergence therefore give
$\|(P_\beta(r)-P_\beta(r_0))g\|_H\to0$ for every $g\in H$.
After extending the integrand by zero to a fixed interval, dominated
convergence gives continuity of the part away from the diagonal.  Hence
\[
        u(t)=S_\beta(t)\varphi+v(t)
\]
belongs to $C([0,T];H)$.  More explicitly,
\[
 \sup_{0\le t\le T}\|u(t)\|_H
 \le C_\beta\|\varphi\|_H
   +\frac{C_\beta T^\beta}{\beta}
      \sup_{0\le s\le T}\|f(s)\|_H,
\]
which proves \eqref{diffusion-stability-estimate} with a constant depending
only on $\beta$ and $T$.

To see uniform convergence of the spectral series, let $Q_N$ be the orthogonal projection onto $\operatorname{span}\{e_n:n>N\}$.  Since $Q_N$ is a spectral projection of $A$,
\[
        Q_NS_\beta(t)=S_\beta(t)Q_N,
        \qquad
        Q_NP_\beta(t)=P_\beta(t)Q_N.
\]
The initial-data tail is therefore bounded uniformly by
$C_\beta\|Q_N\varphi\|$.  Since $f([0,T])$ is compact in $H$ and
$Q_N\to0$ strongly, the convergence is uniform on this compact set:
\[
        \sup_{0\le s\le T}\|Q_Nf(s)\|_H\longrightarrow0.
\]
The forcing tail satisfies
\[
\begin{aligned}
 \left\|Q_N\int_0^tP_\beta(t-s)f(s)\,ds\right\|_H
 &\le C_\beta\int_0^t(t-s)^{\beta-1}
       \|Q_Nf(s)\|_H\,ds\\
 &\le \frac{C_\beta T^\beta}{\beta}
       \sup_{0\le s\le T}\|Q_Nf(s)\|_H.
\end{aligned}
\]
Expanding the spectral multipliers therefore gives \eqref{diffusion-solution}
with uniform convergence in $H$.  The estimates above show that its
right-hand side belongs to $C([0,T];H)$, so it is a mild solution by
Definition \ref{diffusion-mild-definition}.  Since that definition fixes the
variation-of-constants expression uniquely in terms of $\varphi$ and $f$,
uniqueness in the mild class follows.

For each fixed $n$, the corresponding scalar coefficient is the standard
solution of
\[
        {}^{C}D_{0+,t}^{\beta}u_n(t)+\lambda_n u_n(t)=f_n(t),
        \qquad u_n(0)=\varphi_n.
\]
See \cite[Section~3]{SY11}.  This coefficientwise observation does not, under
the present data assumptions, assert termwise reconstruction of an $H$-valued
strong solution.
\end{proof}

\begin{cor}[The positive reference realization]\label{diffusion-positive-reference}
The positivity hypothesis \eqref{diffusion-positive-assumption} is satisfied
by the reference realization determined by
\[
        \xi_1^-(u)=0,
        \qquad
        \xi_2^+(u)=0.
\]
Indeed,
\[
        \mathcal L_{\omega_+}^{-1}=I_0^\alpha I_1^\alpha,
        \qquad
        \mathcal L_{\omega_+}\ge
        \|I_0^\alpha I_1^\alpha\|^{-1}I,
\]
so Theorem \ref{diffusion-mild-theorem} applies with
$A=\mathcal L_{\omega_+}$.
\end{cor}

\begin{rem}
Under stronger assumptions on $\varphi$ and $f$, standard fractional evolution
theory gives strong-solution regularity; see \cite[Chapter~2]{Baz01} and
\cite[Sections~2--3]{SY11}.  This section is included only to illustrate how
the positive spatial realization enters the standard Caputo-time framework;
no new time-fractional well-posedness result is claimed.
\end{rem}

\section{Conclusion}

For the power-kernel expression $\mathcal D_1^\alpha D_0^\alpha$,
$1/2<\alpha<1$, we have defined an endpoint-normalized distributional graph
domain and proved its Volterra representation, including the singular mode
$x^{\alpha-1}$.  The domain admits four bounded graph traces, and its
boundary-minimal zero-trace restriction has the represented graph operator as
its adjoint.  The resulting ordinary boundary triplet gives a uniform
description of all separated and coupled self-adjoint boundary planes within
this realization.

The same representation reduces the zero-eigenvalue question to an explicit
$2\times2$ matrix.  Whenever zero is regular, the inverse is
$I_0^\alpha I_1^\alpha$ plus a boundary correction of rank at most two, with
the correction rank determined exactly by the boundary matrix.  This
finite-rank structure transfers Burman's singular-value estimate to every
invertible self-adjoint restriction and gives a boundary-independent leading
constant, the sharp Schatten threshold, the weak-Schatten endpoint and the
Weyl counting constant.

For the positive reference restriction, the inverse kernel and a sharper
two-term eigenvalue asymptotic were already treated in \cite{CK21}.  The
spectral novelty here is therefore not a new phase asymptotic for that reference
problem, but the extension of the leading singular-value and Weyl laws to the
full self-adjoint boundary family.  The final Caputo-time result is a standard
mild-solution consequence of a positive spatial realization and is included as
an application rather than as an independent time-fractional novelty.

\section*{Statements and Declarations}

\subsection*{Funding}
The author was supported by Grant No. BR31714735 from the Ministry of Science and Higher Education of the Republic of Kazakhstan.

\subsection*{Competing interests}
The author has no relevant financial or non-financial interests to disclose.

\subsection*{Data availability}
No datasets were generated or analyzed during the present study.

\subsection*{Use of generative artificial intelligence}
During the preparation of this work the author used Claude (Anthropic, August 2026) as an editorial and checking aid: for consistency checks of the mathematical exposition, identification of points requiring additional justification, suggestions of relevant literature, and improvement of presentation. All mathematical arguments were verified by the author, and all cited references were checked against the original sources. The author reviewed and edited all output and takes full responsibility for the content of the manuscript.

\end{document}